\documentclass[12pt, a4paper]{article}
\usepackage{amsmath}   
\usepackage{comment}
\usepackage{amsfonts}
\usepackage{amssymb}   
\usepackage{epsfig}
\usepackage{graphicx}
\usepackage{amsthm}
\usepackage{newlfont}
\usepackage{eucal}
\usepackage{eufrak}
\usepackage{color}


\newtheorem{thm}{Theorem}[section]
\newtheorem{lem}[thm]{Lemma}

\newtheorem{pro}[thm]{Proposition}

\newtheorem{con}[thm]{Conjecture}

\theoremstyle{definition}
\newtheorem{den}[thm]{Definition}
\theoremstyle{remark}
\newtheorem*{rem}{Remark}

\renewcommand{\labelenumi}{(\roman{enumi})}
\graphicspath{{figures/}}
\begin{document}
\title{A characterization of regular partial cubes whose all convex cycles have the same lengths}

\footnotetext[1]{The work is partially supported by the National Natural Science Foundation of China (Grant No. 11571155, 12071194, 11961067).}

\author{Yan-Ting Xie$^1$, ~Yong-De Feng$^{2}$, ~Shou-Jun Xu$^1, \thanks{Corresponding author. E-mail address: ~shjxu@lzu.edu.cn (S.-J. Xu).}$}

\date{\small $^1$ School of Mathematics and Statistics, Gansu Center for Applied Mathematics, \\Lanzhou University,
Lanzhou, Gansu 730000, China\\
$^2$ College of Mathematics and Systems Science, Xinjiang University, Urumqi,\\ Xinjiang 830046, China}

\maketitle

\begin{abstract}
Partial cubes are graphs that can be isometrically embedded into hypercubes. Convex cycles play an important role in the study of partial cubes. In this paper, we prove that a regular partial cube is a hypercube (resp., a Doubled Odd graph, an even cycle of length $2n$ where $n\geqslant 4$) if and only if all its convex cycles are 4-cycles (resp., 6-cycles, $2n$-cycles). In particular, the partial cubes whose all convex cycles are 4-cycles are equivalent to almost-median graphs. Therefore, we conclude that regular almost-median graphs are exactly hypercubes, which generalizes the result by Mulder [J. Graph Theory, 4 (1980) 107--110]---regular median graphs are hypercubes.

\setlength{\baselineskip}{17pt}
{} \vskip 0.1in \noindent%
\end{abstract}
\section{Introduction}
A graph is called a {\em partial cube} if it can be embedded into a hypercube isometrically. It was introduced by Graham and Pollak \cite{gp71} as a model for interconnection networks and later for other applications. We refer the readers to two books \cite{hik11, ik00} and a semi-survey \cite{o08} for research on partial cubes.

Regular graphs are graphs that each of its vertices has the same degree. Many important graph classes are subclasses of regular graphs, such that vertex-transitive graphs and distance-regular graphs, etc. Moreover, many models for interconnection networks can be understood as regular graphs, like hypercubes. Since partial cubes are graphs inheriting metric properties from hypercubes, finding the regular partial cubes is an interesting problem. For the general case, Klav\v zar and Shpectorov \cite{ks07} showed that the variety of regular partial cubes, even the variety of cubic ones, is too big to allow an explicit classification, so it is necessary to find the regular partial cubes under some given conditions. In 1992, Weichsel \cite{we92} studied the problem among distance-regular graphs and characterized all distance-regular partial cubes. For vertex-transitive graphs, Marc \cite{mar171} focused on cubic vertex-transitive graphs and gave the characterization of cubic vertex-transitive partial cubes. Then, Marc \cite{mar172} studied mirror graphs, a class of vertex-transitive partial cubes, and proved that the mirror graphs are equivalent to the Cayley graphs of a finite Coxeter group with canonical generators and the tope graphs of a reflection arrangement. Most recently, we studied this problem among Cayley graphs generated by transpositions and proved  that bubble sort graphs are the only graph class of partial cubes in these Cayley graphs \cite{xfx221}. We also characterized regular 2-arc-transitive partial cubes in \cite{xfx222}.

First, Let us define a class of graphs---Doubled Odd graphs. Denote $[n]$ as the set $\{1,2,\cdots,n\}$ and, for an integer $k\leqslant n$, ${[n] \choose k}$ as the set of all subsets with $k$ elements in $[n]$ (referred to as $k$-subsets). For $k\geqslant 1$, the {\em Doubled Odd graph}  \cite{bcn89}, denoted by $\tilde{O}_{k}$, is the bipartite graph with two parts  ${[2k-1] \choose k-1}$ and ${[2k-1] \choose k}$, where two vertices $u\in {[2k-1] \choose k-1}$ and $v\in {[2k-1] \choose k}$ are adjacent if and only if $u\subset v$. It is easy to see that $\tilde{O}_{k}$ is a $k$-regular partial cube.

In a general graph $G$, every cycle of $G$ can be represented as a symmetric difference of isometric cycles \cite{mar18}. For partial cubes, there is a stronger property: convex cycles in a partial cube form a basis for its cycle space \cite{hls13}. Therefore, convex cycles play an important role in the study of partial cubes. Klav\v zar and Shpectorov \cite{ks122} studied the density of convex cycles in partial cubes and obtained an Euler-type inequality related to the number of vertices and edges, isometric dimension, and lengths of every convex cycle.


Motivated by the idea about convex cycles in partial cubes and the fact that all convex cycles in hypercubes (resp., Doubled Odd graphs) are 4-cycles (resp., 6-cycles), we obtain the following result in this paper.
\begin{thm}\label{thm:MainResult}
Let $G$ be a finite, regular partial cube other than $K_1$, $K_2$. Then
\begin{enumerate}
\item all convex cycles of $G$ are 4-cycles if and only if $G$ is isomorphic to a hypercube $Q_k$ ($k\geqslant 2$);
\item all convex cycles of $G$ are 6-cycles if and only if $G$ is isomorphic to a Doubled Odd graph $\tilde{O}_{k}$ ($k\geqslant 2$);
\item all convex cycles of $G$ are $2n$-cycles if and only if $G$ is isomorphic to an even cycle $C_{2n}$ ($n\geqslant 4$).
\end{enumerate}
\end{thm}

A corollary of (i) of Theorem \ref{thm:MainResult} is about almost-median graphs, which is a natural generalization of median graphs. A graph $G$ is called a {\em median graph} if, for every three vertices $u, v, w$ of $G$, there exists a unique vertex, called the {\em median} of $u,v,w$, that lies on a shortest path between each pair of $u,v,w$ simultaneously. 
Median graphs can also be characterized as partial cubes for which, for each edge $uv$, two associated  vertex subsets $U_{uv}$ and $U_{vu}$ are convex (see Section 2 for the detailed definition of $U_{uv}$). {\em Almost-median graphs} (resp., {\em semi-median graphs}) were introduced in \cite{ik98} as those partial cubes for whose all vertex  subsets $U_{uv}$ and  $U_{vu}$ mentioned above are isometric (resp., connected).

Let's denote $\mathcal{H}$, $\mathcal{M}$, $\mathcal{A}$, $\mathcal{S}$, $\mathcal{P}$ as the classes of hypercubes, median graphs, almost-median graphs, semi-median graphs and partial cubes, respectively. Then we have

\begin{equation*}
\mathcal{H}\subsetneq\mathcal{M}\subsetneq\mathcal{A}\subsetneq\mathcal{S}\subsetneq\mathcal{P}.
\end{equation*}

In the study of the relations between the five graph classes, Bre\v sar et al. \cite{bikms02} proved that a semi-median graph is a median graph if and only if it contains no $Q_3^-$ as a convex subgraph, where $Q_3^-$ is the graph obtained by removing one vertex from the cube graph $Q_3$. Bre\v sar \cite{b07} obtained that a semi-median graph is an almost-median graph if and only if all its convex cycles are 4-cycles.  Klav\v zar and Shpectorov extended Bre\v sar's result to the class of partial cubes, i.e.,
\begin{pro}{\em\cite{ks121}}\label{pro:AMG4Cycle}
A partial cube $G$ is an almost-median graph if and only if every convex cycle of $G$ is 4-cycle.
\end{pro}

About the relation between $\mathcal{H}$ and $\mathcal{M}$,  Mulder \cite{mu80a} proved that a median graph $G$ is regular if and only if $G$ is a hypercube in 1980s. Combined with (i) of Theorem \ref{thm:MainResult} and Proposition \ref{pro:AMG4Cycle}, we obtain the following result, which is a generalization of Mulder's result.
\begin{thm}\label{cor:RegularAMGisnCube}
A finite almost-median graph $G$ is $k$-regular if and only if $G\cong Q_k$.
\end{thm}

The paper is organized as follows. In Section 2, we introduce some terminology and lemmas. In Section 3, we prove Theorem \ref{thm:MainResult}. In Section 4, we conclude the paper by suggesting some future problems.

\section{Preliminaries}
In this paper all graphs are undirected, finite and simple. Let $G$ be a graph with vertex set $V(G)$ and edge set $E(G)$. For a vertex $v \in V(G)$, the degree of $v$, denoted by $\mathrm{deg}(v)$, is the number of the neighbors of $v$ in $G$. If $\mathrm{deg}(v)=k$  for all $v\in V(G)$, we say $G$ is {\em $k$-regular}.

The {\em length} of a path (resp., a cycle) is its number of edges. A path (resp., a cycle) of length $l$ is called $l$-path (resp., $l$-cycle). The length of a shortest cycle in $G$ is called the {\em girth} of $G$, denoted by $g(G)$. If $G$ is acyclic, we denote $g(G)=\infty$. Let $P=v_0v_1\cdots v_n$ be a path and $C=u_0u_1\cdots u_mu_0$ a cycle of $G$. Denote $v_iPv_j$ ($0\leqslant i<j\leqslant n$) (resp., $u_iu_{i+1}Cu_j$ ($0\leqslant i,j\leqslant m$)) as the subpath $v_iv_{i+1}\cdots v_j$ (resp., $u_iu_{i+1}\cdots u_j$) on $P$ (resp., $C$).

For $u, v\in V(G)$, the {\em distance} $d_{G}(u,v)$ (we will drop the subscript $G$ if no confusion can occur) is the length of a shortest path between $u$ and $v$ in $G$.  We call a shortest path from $u$ to $v$ a $u,v$-{\em geodesic}. A subgraph $H$ of $G$ is called {\em isometric} if for any $u,v\in V(H)$, $d_H(u,v)=d_G(u,v)$, and further, if for any $u,v\in V(H)$, all $u,v$-geodesics are contained in $H$, we call $H$ a {\em convex} subgraph of $G$. Obviously, the convex subgraph of $G$ is isometric, and the isometric subgraph of $G$ is induced and connected.

A {\em hypercube of dimension $n$} (or {\em $n$-cube} for short), denoted by $Q_n$, is a graph whose vertex set corresponds to the set of 0-1 sequences $a_1a_2\cdots a_n$ with $a_i\in \{0,1\}$, $i=1, 2, \cdots, n$. Two vertices are adjacent if the corresponding 0-1 sequences differ in exactly one digit. A graph $G$ is called a {\em partial cube} if it is isomorphic to an isometric subgraph of $Q_n$ for some $n$.

The {\em Djokovi\'c-Winkler relation} (see \cite{dj73,w84}) $\Theta$ is a binary relation on $E(G)$ defined as follows: Let $e:=uv$ and $f:=xy$ be two edges in $G$, $e\,\Theta\,f\iff d(u,x)+d(v,y)\neq d(u,y)+d(v,x)$. If $G$ is bipartite, there is another equivalent definition of Djokovi\'c-Winkler relation: $e\,\Theta\,f\iff d(u,x)=d(v,y)$ and $d(v,x)=d(u,y)$. Winkler proved that

\begin{pro}{\em\cite{w84}}\label{pro:PartialCube}
A connected graph $G$ is a partial cube if and only if $G$ is bipartite and $\Theta$ is an equivalence relation on $E(G)$.
\end{pro}

Let $G$ be a partial cube. We call the equivalence class on $E(G)$ the {\em $\Theta$-class}. For $e=uv\in E(G)$, let $F_{uv}$ (it can also be written as $F_e$) be the $\Theta$-class containing $uv$. The {\em isometric dimension} of $G$, denoted by $\mathrm{idim}(G)$, is the smallest integer $n$ that $G$ can be embedded into $Q_n$ isometrically, which coincides with the number of $\Theta$-classes of $G$ \cite{dj73}. Denote $W_{uv}:=\{w\in V(G)|d(u,w)<d(v,w)\}$ and $U_{uv}:=\{w\in V(G)|w\in W_{uv}\mbox{ and }w\mbox{ is incident with an edge in}$ $F_{uv}\}$.

The following simple propositions about partial cubes will be useful:
\begin{pro}{\em\cite{mas09}}\label{pro:DWRelationonIsometricCycle}
Let $G$ be a graph and $C$ an isometric even cycle in it. Then each pair of antipodal edges of $C$ are in relation $\Theta$.
\end{pro}
\begin{pro}{\em\cite{ik00}}\label{pro:DWRelationonGeodesic}
Let $G$ be a partial cube with a path $P$ in it. $P$ is a geodesic in $G$ if and only if no two distinct edges on it are in relation $\Theta$. In particular, two adjacent edges in a partial cube cannot be in relation $\Theta$.
\end{pro}
\begin{pro}{\em\cite{xfx222}}\label{pro:ShortestCycleisConvex}
Let $G$ be a partial cube and $C$ a shortest cycle in it. Then $C$ is convex.
\end{pro}

In what follows, we introduce some terminology and a lemma about partial cubes further, which originate from \cite{mar16}.
\begin{den}\cite{mar16}\label{def:Traverse}
Let $v_1u_1\,\Theta\, v_2u_2$ in a partial cube $G$ with $u_2\in U_{u_1v_1}$. Let $C_1,C_2,\cdots,C_n$ be a sequence of isometric cycles such that $u_1v_1$ lies in $C_1$, $u_2v_2$ lies in $C_n$, each pair $C_i$ and $C_{i+1}$ ($1\leqslant i\leqslant n-1$) intersects exactly one edge which is in $F_{u_1v_1}$ and all other pairs don't intersect. If the shortest path from $v_1$ to $v_2$ on the union of $C_1,C_2,\cdots,C_n$ is a shortest $v_1,v_2$-path in $G$, we call the sequence of cycles $T:=(C_1,C_2,\cdots,C_n)$ the {\em traverse} from $u_1v_1$ to $u_2v_2$. We call the shortest $v_1,v_2$-path (resp., $u_1,u_2$-path) on the union of $C_1,C_2,\cdots,C_n$ the {\em $v_1,v_2$-side} (resp., {\em $u_1,u_2$-side}) of the traverse. If all cycles on a traverse are convex, we call the traverse a {\em convex traverse}.
\end{den}
\begin{lem}{\em \cite{mar16}}\label{lem:GeodesicandTraverse}
Let $u_0v_0\,\Theta\, u_mv_m$ hold in partial cube $G$. If $P=u_0u_1u_2\cdots u_m$ is a geodesic, then either $P$ is the $u_0,u_m$-side of a convex traverse from $u_0v_0$ to $u_mv_m$, or there exists a convex cycle $C$ of the form $u_iw_{i+1}w_{i+2}\cdots w_{j-1}u_ju_{j-1}\cdots u_{i+1}u_i$ for some $0\leqslant i<j-1\leqslant m-1$ and the vertices $w_{i+1},w_{i+2},\cdots, w_{j-1}$ not on $P$.
\end{lem}

\begin{rem}
$u_i$, $u_j$ are a pair of antipodal vertices of $C$ in Lemma \ref{lem:GeodesicandTraverse}. We say that $u_iPu_j$ is one {\em half} of $C$ and the other half of $C$ is not on $P$.
\end{rem}

\section{Proof of Theorem \ref{thm:MainResult}}

\subsection{Proof of (i) of Theorem \ref{thm:MainResult}}
Firstly, in order to prove (i) of Theorem \ref{thm:MainResult}, we need some new definitions. Let $G$ be a partial cube and $v$ a vertex in $G$. We define $\mathcal{F}(v)$ as the family of $\Theta$-classes which contain an edge incident with $v$, that is, $\mathcal{F}(v):=\{F_{e}|e\mbox{ is incident with }v\}$.

\begin{den}\label{def:faultypath}
Let $G$ be a partial cube, $P:=v_0v_1\cdots v_l$ ($l\geqslant 1$) a path in $G$ and $e_i:=v_{i-1}v_i$ for $1\leqslant i\leqslant l$. We call $P$ of {\em type-I} if it satisfies that 
$\mathcal{F}(v_0)\setminus \mathcal{F}(v_1)\neq\emptyset$, $\mathcal{F}(v_{l})\setminus \mathcal{F}(v_{l-1})\neq\emptyset$ and $F_{e_{i+1}}\not\in\mathcal{F}(v_{i-1})$, $F_{e_{i}}\not\in\mathcal{F}(v_{i+1})$ for $1\leqslant i\leqslant l-1$. 
\end{den} 

Secondly, we obtain the following lemma about paths of type-I.

\begin{lem}\label{lem:faultyGeodesic}
Let $G$ be a partial cube whose all convex cycles are 4-cycles, $P:=v_0v_1\cdots v_l$ ($l\geqslant 1$) a path in $G$ and $e_i:=v_{i-1}v_i$ for $1\leqslant i\leqslant l$. If $P$ is of type-I, then $P$ is a geodesic.
\end{lem}

\begin{proof}
When $l=1$, it is trivial. Now we prove the lemma when $l\geqslant 2$.

By contradiction, assume $P$ is not a geodesic. By Proposition \ref{pro:DWRelationonGeodesic}, there exist integers $s,t$ such that $e_s\,\Theta\, e_t$ for $1\leqslant s<t\leqslant l$. Choose $s,t$ such that $t-s$ is as small as possible; then $t-s\neq 1$ by Proposition \ref{pro:DWRelationonGeodesic}. In fact, $t-s\neq 2$, otherwise $F_{e_{s}}\in\mathcal{F}(v_{s+1})$ and $F_{e_{s+2}}\in\mathcal{F}(v_{s})$, a contradiction with the definition of paths of type-I. Thus, $t-s\geqslant 3$ and $l\geqslant 4$. By the choice of $s,t$, $F_{e_i}\neq F_{e_j}$ for any $s+1\leqslant i<j\leqslant t-1$, so $P':=v_{s}Pv_{t-1}$ is a $v_{s},v_{t-1}$-geodesic by Proposition \ref{pro:DWRelationonGeodesic}. According to Lemma \ref{lem:GeodesicandTraverse}, we consider two cases.

\textbf{Case 1.}  $P'$ is the $v_{s},v_{t-1}$-side of a convex traverse $T$ from $e_s$ to $e_t$.

In this case, assume $T=(C_1, C_2, \cdots, C_{m})$. Then all of $C_i$ ($1\leqslant i\leqslant m$)  are 4-cycles and further $v_{s-1},v_{s},v_{s+1}$ are on $C_1$. Assume $C_1=v_{s-1}v_sv_{s+1}xv_{s-1}$. By Propositions \ref{pro:DWRelationonIsometricCycle} and \ref{pro:ShortestCycleisConvex}, $e_s\,\Theta\, xv_{s+1}$. Thus, $F_{e_{s}}\in\mathcal{F}(v_{s+1})$, a contradiction with the definition of paths of type-I.

\textbf{Case 2.}  There exists a convex cycle $C$ such that exactly one half of $C$ is on $P'$.

Since $C$ is convex, $C$ is a 4-cycle. Assume $C=v_{p-1}v_pv_{p+1}yv_{p-1}$ for some $s+1\leqslant p\leqslant t-2$.  Then $e_p\,\Theta\, v_{p+1}y$, $e_{p+1}\,\Theta\, v_{p-1}y$. Thus, $F_{e_{p}}\in\mathcal{F}(v_{p+1})$, $F_{e_{p+1}}\in\mathcal{F}(v_{p-1})$, a contradiction with the definition of paths of type-I.
\end{proof}

Finally, we prove (i) of Theorem \ref{thm:MainResult}.

\begin{proof}[Proof (i) of Theorem \ref{thm:MainResult}]
{\em Sufficiency.} Since $Q_k$ is an almost-median graph, every convex cycle of $Q_k$ is a 4-cycle by Proposition \ref{pro:AMG4Cycle}.

{\em Necessity.} Assume $G$ is $k$-regular. Since $G$ is not $K_1$ or $K_2$, we obtain that $k\geqslant 2$. When $k=2$, it is trivial. Now, we prove the case of $k\geqslant 3$. First, we prove $\mathrm{idim}(G)=k$.

Let $v$ be a vertex in $G$. By Proposition \ref{pro:DWRelationonGeodesic}, $|\mathcal{F}(v)|=\mathrm{deg}(v)=k$, so there are at least $k$ different $\Theta$-classes and further $\mathrm{idim}(G)\geqslant k$. By contradiction, assume $\mathrm{idim}(G)>k$, that is, there are at least $k+1$ different $\Theta$-classes. In this case, there must exist two vertices $v_0,v_1$ such that $\mathcal{F}(v_0)\neq\mathcal{F}(v_1)$. Since $G$ is connected, we can assume that $v_0,v_1$ are adjacent. Since $|\mathcal{F}(v_0)|=|\mathcal{F}(v_1)|=k$, we have $\mathcal{F}(v_0)\setminus \mathcal{F}(v_1)\neq\emptyset$, $\mathcal{F}(v_0)\setminus \mathcal{F}(v_1)\neq\emptyset$, and further the path $v_0v_1$ is a path of type-I of length 1. Therefore, paths of type-I must exist in $G$.

By Lemma \ref{lem:faultyGeodesic}, every path of type-I is a geodesic. Let $P:=v_0v_1\cdots v_l$ ($l\geqslant 1$) be a longest geodesic of type-I in $G$ and $e_i:=v_{i-1}v_i$ for $1\leqslant i\leqslant l$. Since $G$ is finite, $l$ is finite. By the definition of paths of type-I, there exists an edge $e$ incident with $v_l$ where $F_e\not\in\mathcal{F}(v_{l-1})$. Let the other endpoint of $e$ be $v_{l+1}$. Then $v_{l+1}\neq v_i$ ($0\leqslant i\leqslant l-2$) because $P$ is a geodesic. By our hypothesis, the path $v_0v_1\cdots v_lv_{l+1}$ is not of type-I, then either $\mathcal{F}(v_{l+1})=\mathcal{F}(v_l)$, or $\mathcal{F}(v_{l+1})\neq\mathcal{F}(v_l)$ but $F_{e_{l}}\in\mathcal{F}(v_{l+1})$. Whichever the case is, $F_{e_{l}}\in\mathcal{F}(v_{l+1})$. Let $e'$ be the edge incident with $v_{l+1}$ such that $e'\,\Theta\, e_{l}$, and $v_{l+2}$ the other endpoint of $e'$.  Then $d(v_{l-1},v_{l+2})=d(v_l,v_{l+1})=1$, that is, $v_{l-1},v_l,v_{l+1},v_{l+2}$ compose a 4-cycle and further $e\,\Theta\, v_{l-1}v_{l+2}$, yielding the contradiction of $F_e\not\in\mathcal{F}(v_{l-1})$. Thus, for all $v\in V(G)$, $\mathcal{F}(v)$ are equal, and further $\mathrm{idim}(G)=k$.

Thus, $G$ can be embedded in $Q_k$ isometrically. Since $Q_k$ is also $k$-regular, $G\cong Q_k$.
\end{proof}
\subsection{Proof of (ii) of Theorem \ref{thm:MainResult}}
Summarizing (ii) of Theorem 3.2 in \cite{xfx222} and its proof, we obtain the following lemma:
\begin{lem}{\em\cite{xfx222}}\label{lem:TwoConditions}
Let $G$ be a finite, $k$-regular partial cube with $g(G)=6$. Then $G\cong \tilde{O}_k$ if and only if the following statements hold:
\begin{enumerate}
\renewcommand{\labelenumi}{\em (\alph{enumi})}
\item all convex cycles of $G$ are 6-cycles;
\item every 3-path of $G$ is in a 6-cycle.
\end{enumerate}
\end{lem}

In fact, as the following lemma shows, the statements (a) and (b)  are equivalent under the conditions of Lemma \ref{lem:TwoConditions}. 
\begin{lem}\label{lem:EquivalentConditions}
Let $G$ be a finite, $k$-regular partial cube with $g(G)=6$. Then the following statements are equivalent:
\begin{enumerate}
\renewcommand{\labelenumi}{\em (\alph{enumi})}
\item all convex cycles of $G$ are 6-cycles;
\item every 3-path of $G$ is in a 6-cycle.
\end{enumerate}
\end{lem}

Before proving Lemma \ref{lem:EquivalentConditions}, we give a definition of paths of type-II, similar to paths of type-I.

\begin{den}\label{def:ModFaultyPath}
Let $G$ be a partial cube, $P:=v_0v_1\cdots v_l$ ($l\geqslant 3$) a path in $G$ and $e_i:=v_{i-1}v_{i}$ for $1\leqslant i\leqslant l$. We call $P$ of {\em type-II} if it satisfies that $F_{e_i}\not\in\mathcal{F}(v_{i+2})$ and $F_{e_{i+2}}\not\in\mathcal{F}(v_{i-1})$ for $1\leqslant i\leqslant l-2$.
\end{den}

About paths of type-II, we obtain the following lemma. 
\begin{lem}\label{lem:MFPathareGeodesic}
Let $G$ be a partial cube whose all convex cycles are 6-cycles and $P:=v_0v_1\cdots v_l$ ($l\geqslant 3$) a path in $G$. If $P$ is a path of type-II, then $P$ is a geodesic.
\end{lem}
\begin{proof}
Note that if all convex cycles of $G$ are 6-cycles, then $g(G)=6$ by Proposition \ref{pro:ShortestCycleisConvex}.

When $l=3$, $P$ is a geodesic by $g(G)=6$. Now, let's assume $l\geqslant 4$.

Firstly, we prove that no subpath of $P$ of length 3 is in a 6-cycle. By contradiction, assume the 3-path $v_iv_{i+1}v_{i+2}v_{i+3}$ ($0\leqslant i\leqslant l-3$) in a 6-cycle, denote as $v_iv_{i+1}v_{i+2}v_{i+3}yxv_i$. Then $v_{i+3}y\,\Theta\, v_{i+1}v_i$ and $v_{i}x\,\Theta\, v_{i+2}v_{i+3}$ by Propositions \ref{pro:DWRelationonIsometricCycle} and \ref{pro:ShortestCycleisConvex}, a contradiction with the definition of paths of type-II.

By contradiction, assume $P$ is not a geodesic. By Proposition \ref{pro:DWRelationonGeodesic}, there exist $e_s, e_t$ satisfying $e_s\,\Theta\,e_t$ for $1\leqslant s<t\leqslant l$. We select $s, t$ such that $e_s\,\Theta\,e_t$ and $t-s$ is as small as possible. If $t-s\leqslant 2$, it is a contradiction with Proposition \ref{pro:DWRelationonGeodesic} or $g(G)=6$. Thus, $t-s\geqslant 3$. By the choice of $s,t$, $F_{e_i}\neq F_{e_j}$ for any $s+1\leqslant i<j\leqslant t-1$, so $P':=v_{s}Pv_{t-1}$ is a $v_{s},v_{t-1}$-geodesic by Proposition \ref{pro:DWRelationonGeodesic}. According to Lemma \ref{lem:GeodesicandTraverse}, we consider two cases.

\textbf{Case 1.}  $P'$ is the $v_{s},v_{t-1}$-side of a convex traverse $T$ from $e_s$ to $e_t$.

Assume $T=(C_1, C_2, \cdots, C_{m})$ where $C_i$ is 6-cycle for any $1\leqslant i\leqslant m$. Then the subpath $v_{s-1}v_sv_{s+1}v_{s+2}$ is in $C_1$, a contradiction.

\textbf{Case 2.}  There exists a convex cycle $C$ such that exactly one half of $C$ is on $P'$.

Since $C$ is convex, it is a 6-cycle. The half of $C$ on $P'$ is a subpath of $P$ of length 3, a contradiction.

In conclusion, $P$ is a geodesic.
\end{proof}


Now, we prove of Lemma \ref{lem:EquivalentConditions}.
\begin{proof}[Proof of Lemma \ref{lem:EquivalentConditions}]

(b)$\Longrightarrow$(a). By contradiction, assume $G$ contains a convex cycle $C=v_0v_1v_2\cdots v_{m-1}v_0$ where $m\geqslant 8$ is an even number. By hypothesis, the 3-path $v_0v_1v_2v_3$ is in a 6-cycle, denoted by $C'=v_0v_1v_2v_3u_4u_5v_0$. Since $v_0v_1v_2v_3$ is in the convex cycle $C$, it is a $v_0,v_3$-geodesic. Then 3-path $v_0u_5u_4v_3$ is also a $v_0,v_3$-geodesic. Since $C$ is convex, we obtain $C=C'$, a contradiction. 

(a)$\Longrightarrow$(b). Assume $G$ is a finite, $k$-regular partial cube $(k\geqslant 2)$ with $g(G)=6$. When $k=2$, $G\cong C_6$ and it is trivial. Now, we set $k\geqslant 3$. By contradiction, assume $v_0v_1v_2v_3$ is a 3-path of $G$ not in any 6-cycles. Then $F_{v_0v_1}\not\in\mathcal{F}(v_3)$ and $F_{v_2v_3}\not\in\mathcal{F}(v_0)$, that is, $v_0v_1v_2v_3$ is a path of type-II. Thus, there exists at least one path of type-II in $G$. Since $G$ is finite and the paths of type-II are geodesics by Lemma \ref{lem:MFPathareGeodesic}, we can choose $P:=v_0v_1\cdots v_l$ ($l\geqslant 3$) as a longest path of type-II and set $e_i:=v_{i-1}v_{i}$ for $1\leqslant i\leqslant l$. Obviously, $l$ is finite. By the definition of paths of type-II, $F_{e_{l-2}}\not\in\mathcal{F}(v_l)$. And further $F_{e_{l-1}}\not\in\mathcal{F}(v_l)$ otherwise $v_{l-2},v_{l-1},v_l$ are in a 4-cycle, a contradiction with $g(G)=6$. Combined with the three conditions that $F_{e_{l-2}}\not\in\mathcal{F}(v_l)$, $F_{e_{l-1}}\not\in\mathcal{F}(v_l)$ and $|\mathcal{F}(v_{l-2})|=|\mathcal{F}(v_l)|=k$, there must exist an edge $e$ incident with $v_l$ besides $e_l$, which satisfies $F_e\not\in\mathcal{F}(v_{l-2})$. Let $v_{l+1}$ be the other endpoint of $e$. Then $v_{l+1}\neq v_i$ ($0\leqslant i\leqslant l-1$) since $P$ is a geodesic. Since $P$ is the longest path of type-II, the path $v_0v_1\cdots v_lv_{l+1}$ is not of type-II. By the definition of paths of type-II, we obtain that $F_{e_{l-1}}\in\mathcal{F}(v_{l+1})$. Let $e'$ be the edge incident with $v_{l+1}$ satisfying that $e_{l-1}\,\Theta\,e'$ and the other endpoint $v_{l+2}$. Since $P$ is a geodesic and $g(G)=6$, $v_{l+2}\neq v_i$ ($0\leqslant i\leqslant l$). Then $d(v_{l-2},v_{l+2})=d(v_{l-1},v_{l+1})=2$. Set the common neighbor of $v_{l-2},v_{l+2}$ is $v_{l+3}$. Similarly, $v_{l+3}\neq v_i$ ($0\leqslant i\leqslant l$ and $i\neq l-3$). Thus, $v_{l-2}v_{l-1}v_lv_{l+1}v_{l+2}v_{l+3}v_{l-2}$ is a 6-cycle, so it is convex by Proposition \ref{pro:ShortestCycleisConvex} and further $e\,\Theta\, v_{l+3}v_{l-2}$ by Proposition \ref{pro:DWRelationonIsometricCycle}, a contradiction with $F_e\not\in\mathcal{F}(v_{l-2})$. Therefore, every 3-path of $G$ is in a 6-cycle. 
\end{proof}

Finally, we prove (ii) of Theorem \ref{thm:MainResult}.
\begin{proof}[Proof of (ii) of Theorem \ref{thm:MainResult}]
It is obvious that $g(G)=6$ if all convex cycles of $G$ are 6-cycles or $G\cong \tilde{O}_k$, so (ii) can be obtained directly from Lemmas \ref{lem:TwoConditions} and \ref{lem:EquivalentConditions}.
\end{proof}
\subsection{Proof of (iii) of Theorem \ref{thm:MainResult}}
Marc obtained the following lemma.
\begin{lem}{\em\cite{mar16}}\label{lem:RegularPCwithGirth>6}
Let $\Gamma$ be a finite, regular partial cube with $g(\Gamma)>6$, Then $\Gamma$ is $K_1$, $K_2$ or even cycle $C_{2n}$ ($n\geqslant 4$).
\end{lem}

Now, we prove (iii) of Theorem \ref{thm:MainResult}.
\begin{proof}[Proof of (iii) of Theorem \ref{thm:MainResult}]
By Proposition \ref{pro:ShortestCycleisConvex}, $g(\Gamma)=2n>6$, thus (iii) can be obtained directly from Lemma \ref{lem:RegularPCwithGirth>6}.
\end{proof}

\section{Conclusion and problems}
In this paper, we characterize the regular partial cubes whose all convex cycles have the same lengths as hypercubes, Doubled Odd graphs and even cycles of lengths at least 8. As a corollary, we obtain that regular almost-median graphs are exactly hypercubes, extending Mulder's result \cite{mu80a} that regular median graphs are hypercubes.

In \cite{bikms02}, Bre\v sar et al. introduced a class of graphs called {\em tiled graphs}: Let $H_1,H_2,\cdots,$ $H_n$ be subgraphs of a graph $G$. The {\em symmetric sum} $H_1\oplus H_2\oplus\cdots\oplus H_n$ is the subgraph of $G$ induced by the edges of $G$ that appear in an odd number of the subgraphs $H_1,H_2,\cdots,H_n$. Let $C$ be a cycle of $G$. Then a set of 4-cycles $\mathcal{C}=\{C_1,C_2,\cdots,C_n\}$ is called the {\em tiling} of $C$ if $C=C_1\oplus C_2\oplus\cdots\oplus C_n$. A graph $G$ is called {\em tiled} if every cycle of $G$ has a tiling. It was proved that almost-median graphs are tiled partial cubes and tiled partial cubes are semi-median graphs \cite{bikms02}. A natural problem is whether regular tiled partial cubes or regular semi-median graphs are exactly hypercubes.

For the partial cubes with girth 6, Marc  conjectured that
\begin{con}{\em\cite{mar171}}\label{con:VTPCwithGirth6}
Doubled Odd graphs are the only vertex-transitive partial cubes with girth 6. 
\end{con}

In fact, by Lemmas \ref{lem:TwoConditions} and \ref{lem:EquivalentConditions}, we obtain the following theorem:
\begin{thm}\label{thm:ThreeConditions}
Let $G$ be a finite, $k$-regular partial cube with $g(G)=6$. Then the following statements are equivalent:
\begin{enumerate}
\renewcommand{\labelenumi}{\em (\alph{enumi})}
\item all convex cycles of $G$ are 6-cycles;
\item every 3-path of $G$ is in a 6-cycle;
\item $G\cong \tilde{O}_k$.
\end{enumerate}
\end{thm}

Therefore, the future studies can consider the following question: does (a) or (b) in Theorem \ref{thm:ThreeConditions} hold true for every vertex-transitive partial cube $G$ with girth 6?

\vskip 0.2 cm
\noindent {\bf Acknowledgements:} The authors thank the referees for their careful reviews. This work is partially supported by National Natural Science Foundation of China (Grants No. 12071194, 11571155, 11961067).


\begin{thebibliography}{99}
\parskip -0.1cm



\bibitem{b07}B. Bre\v sar, Characterizing almost-median graphs, European J. Combin., 28 (2007) 916--920.

\bibitem{bikms02}B. Bre\v sar, W. Imrich, S. Klav\v zar, H. M. Mulder, R. \v Skrekovski, Tiled partial cubes, J. Graph Theory, 40 (2002) 91--103.

\bibitem{bcn89}A. E. Brouwer, A. M. Cohen, A. Neumaier, Distance-Regular Graphs, Springer-Verlag, Berlin, Heidelberg, New York, 1989.

\bibitem{dj73}D. \v{Z}. Djokovi\'c, Distance preserving subgraphs of hypercubes, J. Combin. Theory Ser. B, 14 (1973) 263--267.

\bibitem{gp71}R. L. Graham, H. Pollak, On the addressing problem for loop switching, Bell System Tech. J., 50 (1971) 2495--2519.

\bibitem{hik11}R. Hammack, W. Imrich, S. Klav\v{z}ar, Handbook of product graphs, Boca Raton, CRC press, 2011.

\bibitem{hls13}M. Hellmuth, J. Leydold, P. Stadler, Convex cycle bases, Ars Math. Contemp., 7 (2013) 123--140.

\bibitem{ik98}W. Imrich, S. Klav\v zar, A convexity lemma and expansion procedures for bipartite graphs, European J. Combin., 19 (1998) 677--685.

\bibitem{ik00}W. Imrich, S. Klav\v zar, Product Graphs: Structure and Recognition, John Wiley \& Sons, New York, USA, 2000.


\bibitem{ks07}S. Klav\v zar, S. Shpectorov, Tribes of cubic partial cubes, Discrete Math. Theor. Comput. Sci., 9 (2007), 273--292.

\bibitem{ks121}S. Klav\v zar, S. Shpectorov, Characterizing almost-median graphs II, Discrete Math., 312 (2012) 462--464.

\bibitem{ks122}S. Klav\v zar, S. Shpectorov, Convex excess in partial cubes, J. Graph Theory, 69 (2012) 356--369.

\bibitem{mar16}T. Marc, There are no finite partial cubes of girth more than 6 and minimum degree at least 3, European J. Combin., 55 (2016) 62--72.

\bibitem{mar171}T. Marc, Classification of vertex-transitive cubic partial cubes, J. Graph Theory, 86 (2017)  406--421.

\bibitem{mar172}T. Marc, Mirror graphs: graph theoretical characterization of reflection arrangements and finite coxeter groups, European J. Combin., 63 (2017) 115--123.

\bibitem{mar18}T. Marc, Cycling in hypercubes, PhD thesis, Univ. of Ljubljana, 2018.

\bibitem{mas09}M. Massow, Linear extension graphs and linear extension diameter, Dissertation, Technische Universit\"at Berlin, 2009.


\bibitem{mu80a}H. M. Mulder, $n$-cubes and median graphs, J. Graph Theory, 4 (1980) 107--110.



\bibitem{o08}S. Ovchinnikov, Partial cubes: structures, characterizations, and constructions. Discrete Math., 308 (2008)  5597--5621.

\bibitem{w84}P. M. Winkler, Isometric embedding in products of complete graphs, Discrete Appl. Math., 7 (1984) 221--225.

\bibitem{we92}P. M. Weichsel, Distance regular subgraphs of a cube, Discrete Math., 109 (1992) 297--306.

\bibitem{xfx221}Y.-T. Xie, Y.-D. Feng, S.-J. Xu, Hypercube embeddings and Cayley graphs generated by transpositions, Math. Meth. Appl. Sci., 45 (2022) 7227--7237.

\bibitem{xfx222}Y.-T. Xie, Y.-D. Feng, S.-J. Xu, Characterization of 2-arc-transitive partial cubes, Discrete Math., 346 (2023) 113190.
\end{thebibliography}
\end{document}